\documentclass[11pt]{amsart}%
\usepackage{color}
\usepackage{amsmath}
\usepackage{amssymb}
\usepackage{amsfonts}
\usepackage[french,english]{babel}
\usepackage[latin1]{inputenc}
\usepackage{graphicx}%
\providecommand{\U}[1]{\protect\rule{.1in}{.1in}}
\providecommand{\U}[1]{\protect\rule{.1in}{.1in}}

\DeclareMathSymbol{\subsetneqq}{\mathbin}{AMSb}{36}

\theoremstyle{plain}
\numberwithin{equation}{section}
\newtheorem{theorem}{Theorem}[section]
\newtheorem{corollary}{Corollary}[section]
\newtheorem{lemma}{Lemma}[section]
\newtheorem{proposition}{Proposition}[section]
\newtheorem{definition}{Definition}[section]
\newtheorem{remark}{Remark}[section]
\newtheorem{notation}{Notation}
\begin{document}
\title[Navier-Stokes equations in $L^{3,\infty}(\mathbb{R}^3)$]
{On the uniqueness and the regularity of solutions to the Navier-Stokes in the Lorentz space $L^{3,\infty}(\mathbb{R}^3)$}%
\author{Ramzi May}%
\address{Department of Mathematics and Statistics, College of Science, King Faisal University, Al-Ahsa, Kingdom of Saudi Arabia}
\email{rmay@kfu.edu.sa}

\begin{abstract}
We prove uniqueness of solutions to the Navier-Stokes equations in the space
\[
L^p\bigl([0,T],\widetilde{L}^{3,\infty}(\mathbb{R}^3)\bigr)
\cap
C\bigl([0,T],B^{-1,\infty}_{\infty}(\mathbb{R}^3)\bigr),
\qquad p>2,
\]
where $\widetilde{L}^{3,\infty}(\mathbb{R}^3)$ denotes the closure of the Schwartz space $\mathcal{S}(\mathbb{R}^3)$ in the Lorentz space $L^{3,\infty}(\mathbb{R}^3)$. Moreover, we prove that any solution $u$ in this class, if it exists, is smooth for positive time, namely,  $u\in C^\infty((0,T]\times \mathbb{R}^3)$. In addition, we establish precise estimates for the behavior of $u(t)$ as $t\to0^+ $  in various Banach spaces.
\end{abstract}
\maketitle
\section{Introduction and  main result}
The incompressible Navier--Stokes equations in the whole space $\mathbb{R}^3$ are given by
\begin{equation}
\left\{
\begin{array}{l}
u_{t}-\Delta u+\mathbb{P}\nabla .(u\otimes u)=0 \\
\nabla .u=0 \\
u(0,x)=u_{0}(x)%
\end{array}%
\right.   \tag{NS}
\end{equation}%
where $u=u(t,x)$ denotes the velocity field, $u_0$ is the initial velocity, and $\mathbb{P}$ is the Leray projection onto divergence-free vector fields. We refer the reader to \cite{Lem3,Can} for background on the Navier-Stokes equations and their mathematical and physical interpretations.
One of the most important properties of the Navier-Stokes equations is  the
following scaling invariance property: If $u$ is a solution to the equations (NS)
then for any $\lambda >0,$ the function $u_{\lambda }(t,x)=\lambda u(\lambda
^{2}t,\lambda x)$ is a solution to the Navier-Stokes associated to the
initial data $u_{0,\lambda }(x)=\lambda u_{0}(\lambda x).$ This property leads to the introduction of the notion of critical spaces for the equations (NS): A functional
space $X$ for which $ \mathcal{S}(\mathbb{R}^3)\hookrightarrow X\hookrightarrow \mathcal{S}^\prime(\mathbb{R}^3)$
is called a critical space for the Navier-Stokes equations if its norm is
invariant under the transformation $u_{0}\mapsto u_{0,\lambda }$in the sense%
\[
\left\Vert u_{0,\lambda }\right\Vert _{X}=\left\Vert u_{0}\right\Vert
_{X}~\forall u_{0}\in X,\forall ~\lambda >0.
\]%
Important examples of critical spaces are the homogeneous Sobolev space $\dot{H}^{%
\frac{1}{2}}(\mathbb{R}^{3}),$ the Lebesgue space $L^{3}(\mathbb{R}^{3})$,
and the Lorentz space $L^{3,\infty }(\mathbb{R}^{3}).$
One approach to deal with the existence of solutions to the equations (NS) consists in reformulating these equations into a fixed point problem. In fact according to the Duhamel principle, the equations (NS) can be written as the integral equation
\begin{equation}
u(t)=e^{t\Delta }u_{0}-B(u,u)(t)  \label{eq2}
\end{equation}
where $(e^{t\Delta })_{t\geq 0}$ is the heat semigroup, and $B$ is the
bilinear operator defined as follows%
\begin{equation}
B(u,v)(t)=\int_{0}^{t}\mathbb{P}\nabla e^{(t-s)\Delta }(u\otimes v)(s)ds.
\label{eq3}
\end{equation}%
Then by establishing  the continuity of the bilinear operator $B $ on an appropriate Banach space associated to the initial data $u_0$, the existence of regular solutions to the equations (NS) can be deduced from the application of the Banach fixed point theorem. For more details, we refer the reader to the pioneering
works \cite{FK} and \cite{Kato}, and the books \cite{BCD, Can, Lem2} where this approach is followed to prove the existence of regular solutions to the equations of Navier-Stokes for initial values in different critical spaces.
In 1998, Furioli, Lemarie-Rieusset and Teraneo \cite{FLT1} have proved the uniqueness of
the solutions to the Navier-Stokes equations in the space $C([0,T],L^{3}(%
\mathbb{R}^{3})).$ The importance of this result comes in part from the fact
that the bilinear operator $B$ is not continuous from $C([0,T],L^{3}(\mathbb{%
R}^{3}))\times C([0,T],L^{3}(\mathbb{R}^{3}))$ into $C([0,T],L^{3}(\mathbb{R}%
^{3}))$ as it was shown later by Oru \cite{Oru}. Just later, Yves Meyer
\cite {Mey} established an unexpected result: the bilinear operator $B$ is continuous from  $B$ from $L^\infty([0,T],L^{3,\infty }(\mathbb{R}^{3}))\times L^\infty([0,T],L^{3,\infty }(%
\mathbb{R}^{3}))$ into $L^\infty([0,T],L^{3,\infty }(\mathbb{R}^{3}))$. This result provides  a direct proof of the uniqueness result of Furioli, Lemarie-Rieusset and Teraneo. An another important feature of  $L^{3,\infty }(\mathbb{R}^{3})$ is that, unlike $L^{3}(%
\mathbb{R}^{3})$, it contains nontrivial homogeneous functions. For example $f(x)=\left\vert x\right\vert^{-1}$ belongs to $L^{3,\infty }(\mathbb{R}^{3})$. This makes it a natural setting for the study of self-similar solutions; see, for instance, \cite{Lem2, Ler,Pla}.
\par\noindent In 2007, Lemarie-Rieusset \cite{Lem}, in a technically involved paper improved a uniqueness result due to  Yves Chemin \cite{Che}, by establishing the uniqueness of the solutions to the equations (NS) in the space $%
L^{p}([0,T],L^{q}(\mathbb{R}^{3}))\cap C([0,T],B_{\infty }^{-1,\infty }(%
\mathbb{R}^{3}))$ provided that $p>2$ and $q>3.$ Later, in \cite{May2}, we extended this result to the limiting case $q=3$. This provides an other
independent proof of the original Furioli, Lemarie-Rieusset and Teraneo
uniqueness result since $L^{3}(\mathbb{R}^{3})$ is a subspace of the limit Besov
space $B_{\infty }^{-1,\infty }(\mathbb{R}^{3}),$ (see Lemma \ref{Linj} in the next section). The results described above naturally lead to the question of whether the uniqueness result established in \cite{May2} remains valid when the spatial space $L^3(\mathbb{R}^3)$ is replaced by the larger Lorentz space $L^{3,\infty}(\mathbb{R}^3)$. The main purpose of this paper is to provide an almost affirmative answer to this question. In fact,  we prove the uniqueness in the space
$$L^p\bigl([0,T],\widetilde{L}^{3,\infty}(\mathbb{R}^3)\bigr)
\cap
C\bigl([0,T],B^{-1,\infty}_{\infty}(\mathbb{R}^3)\bigr),
\qquad p>2,$$
where $\widetilde{L}^{3,\infty}(\mathbb{R}^3)$  is
closure of the Schwartz space $\mathcal{S}(\mathbb{R}^3)$ in the Lorentz space $L^{3,\infty}(\mathbb{R}^3)$. Moreover, we establish some regularity proprieties of solutions to the equations (NS) that belong to these spaces. To state precisely our main result, we need to introduce the following notation that will be used frequently throughout this paper.
\begin{notation}
 If $X$ is a Banach space, $s>0$, and $0<T<\infty$, we denote by $C_{0,\frac{s}{2}}([0,T],X)$ the space of functions $v\in C((0,T],X)$ such that $\lim_{t\to 0^+}t^{\frac{s}{2}}\left\Vert v(t)\right\Vert_X=0.$
\end{notation}
Our main theorem sates as follows:
\begin{theorem}\label{main}
let $T>0$ and $p>2$. The equations of Navier-Stokes have at most one solution  in the space $L^p([0,T],\tilde{L}^{3,\infty}(\mathbb{R}^3))\cap C([0,T],B^{-1,\infty}_\infty(\mathbb{R}^3)).$ Moreover, if $u$ is a such solution then it satisfies the following regularity properties:
\begin{enumerate}
    \item $u\in C^\infty((0,T]\times \mathbb{R}^3).$
    \item For every $s>0, u\in C_{0,\frac{s+1}{2}}([0,T],B^{s,\infty}_\infty(\mathbb{R}^3))$.
    \item $ u\in C_{0,\frac{1}{2}}([0,T],L^\infty(\mathbb{R}^3))$.
    \item For every $\alpha>0, \sqrt{-\Delta}^\alpha u\in C_{0,\frac{\alpha}{2}}([0,T],L^{3,\infty}(\mathbb{R}^3))$.
\end{enumerate}
\end{theorem}
\begin{remark}
The uniqueness assertion in this theorem can be deduced from \cite[Theorem 1.3]{May2}, since
\[
L^{3,\infty}(\mathbb{R}^3)\subset M^{r,3}(\mathbb{R}^3),
\qquad r\in(2,3).
\]
However, the proof presented here is more transparent  and avoids much of the technicality involved in the argument of \cite{May2}.
\end{remark}
The remainder of the paper is organized as follows. Section 2 is devoted to recalling the preliminary results needed in Section 3, where we prove Theorem \ref{main} and establish several intermediate results that are of independent interest.
\section{Preliminaries}
In this section, we collect the preliminary results and notations that will be used throughout the paper. We first introduce the functional spaces relevant to our analysis and recall some of their basic properties. We then present, following Kato's approach, the existence and regularity results for a class of solutions to the Navier--Stokes equations that will be needed later. Most of the material in this section is classical. For the convenience of the reader, we include proofs of some auxiliary results and refer to \cite{BCD,BL,Can,Lem2} for further details.
\par\noindent Before starting, let us recall some standard notations that will be used throughout the paper.
\begin{notation}
\begin{enumerate}
    \par\noindent\item $A\lesssim B$ means that there exists a positive constant $c$ such that $A\leq c B$.
    \item $A\simeq B$ means $A\lesssim B$ and $B\lesssim A$.
    \item If $\Theta:\mathbb{R}^n\to\mathbb{R}$, we denote by $\Theta(D)$ the Fourier multiplier with symbol $\Theta(\xi)$, i.e.,
    \[
    \Theta(D)(f)=\mathcal{F}^{-1}\left(\Theta(\xi)\mathcal{F}(f)\right),~\forall f\in \mathcal{S}^\prime(\mathbb{R}^n),
    \]
where $\mathcal{F}$ and $\mathcal{F}^{-1}$ are respectively the Fourier transform and its inverse. In particular $\sqrt{-\Delta}^\alpha$ is the Fourier multiplier with symbol $\left\vert\xi\right\vert^\alpha$ where $\left\vert . \right\vert$ is the Euclidean norm on $\mathbb{R}^n.$
\end{enumerate}
\end{notation}
\subsection{Functional spaces and their properties}
Although we consider the Navier Stokes equations in $\mathbb{R}^3$, for the sake of clarity, we define in this preliminary section the relevant functional spaces in the general setting of $\mathbb{R}^n$. We first introduce a class of Banach spaces that includes the Lebesgue spaces $L^p(\mathbb{R}^n)), (1\leq p\leq\infty)$.
\begin{definition}\label{reg} [admissible spaces]
A Banach space $X$ is called admissible space if it satisfies the following three
properties:
\begin{enumerate}
\item $\mathcal{S}(\mathbb{R}^{n})\hookrightarrow X\hookrightarrow \mathcal{S}^\prime(\mathbb{R}^{n})$.
\item $X\subset L^1_{\text{loc}}(%
\mathbb{R}^{n}).$
\item For every $f\in X$ and $x_0\in \mathbb{R}^{n}$,$f(.+x_0)\in X$ and $\left\Vert f(.+x_0)\right\Vert _{X}=\left\Vert  f\right\Vert _{X}.$
\item For every $f\in X$ and $h\in L^\infty(\mathbb{R}^{n})$, $hf\in X$ and $\left\Vert hf\right\Vert _{X}\lesssim\left\Vert h\right\Vert
_{\infty }\left\Vert f\right\Vert _{X}$.
\end{enumerate}
We denote by $\tilde{X}$ the  closure of $\mathcal{S}(\mathbb{R}^{n})$ in $X$.
\end{definition}
\begin{remark}
It follows readily from properties (1) and (3) in the preceding definition (see, for instance, \cite[Proposition 4.1]{Lem2}) that every admissible space $X$ satisfies the convolution estimate
  $$ \left\Vert g\ast f\right\Vert _{X}\leq\left\Vert g\right\Vert
_{1 }\left\Vert f\right\Vert _{X}$$
for every $f\in X$ and $g\in L^1(\mathbb{R}^{n})$.
\end{remark}
We next recall the Littlewood-Paley decomposition, which plays a fundamental role in the definition of Besov spaces and in the analysis carried out throughout this paper.
\begin{definition}[The Littlewood-Palay decomposition]
Let $\varphi \in D(\mathbb{R}^{n})$ be a non-negative radial function such that $%
\varphi =1$ on the closed ball $B(0,\frac{1}{2})$ and $\varphi =0$ outside $B(0,1).$ Let $%
\psi $ be the function defined by $\psi (\xi )=\varphi (\frac{\xi }{2}%
)-\varphi (\xi ).$ For every non negative integer $j$, we define the
operators $S_{j}$ and $\Delta _{j}$ as the Fourier multipliers with symbols $\varphi (%
\frac{\xi}{2^{j}})$ and $\psi (\frac{\xi}{2^{j}})$ respectively. For every non
negative integer $N$ and  $f\in \mathcal{S}^{\prime }(\mathbb{R}^{n}),$ the identity%
\[
f=S_{N}f+\sum_{j\geq N}\Delta _{j}f
\]%
holds in $\mathcal{S}^{\prime }(\mathbb{R}^{n}).$ This identity is called the
Littlewood-Paley decomposition of the distribution $f.$
\end{definition}

\begin{notation}, we define the operators $S_{j}$ and $%
\Delta _{j}$ for negative integers as follows:%
\begin{eqnarray*}
S_{j}f &=&0, \\
\Delta _{j}f &=&\left\{
\begin{array}{l}
S_{0}f,~j=-1 \\
0,~j<-1%
\end{array}%
\right.
\end{eqnarray*}%
for every $f\in \mathcal{S}^{\prime }(\mathbb{R}^{n}).$
\end{notation}
We are now in position to introduce the notion of non homogeneous Besov spaces $B_{X}^{s,q}(\mathbb{R}^{n})$ where $X$ is a admissible space, $s\in\mathbb{R}$, and $1\leq q\leq\infty$.
\begin{definition}
Let $X$ be an admissible space, $1\leq q\leq \infty $ and $s\in \mathbb{R}.$ The
Besov space $B_{X}^{s,q}(\mathbb{R}^{n})$ is the set of all $f\in \mathcal{S}^{\prime
}(\mathbb{R}^{n})$ such that $\Delta _{j}f\in X$ for every $j\geq -1$ and $%
\left\Vert f\right\Vert _{B_{X}^{s,q}}<\infty $ where
\[
\left\Vert f\right\Vert _{B_{X}^{s,q}}=\left\{
\begin{array}{l}
\left( \sum_{j=-1}^{\infty }\left( 2^{sj}\left\Vert \Delta _{j}f\right\Vert
_{X}\right) ^{q}\right) ^{1/q}\text{, if }q<\infty , \\
\sup_{j\geq -1}2^{sj}\left\Vert \Delta _{j}f\right\Vert _{X}\text{, if }%
q=\infty .%
\end{array}%
\right.
\]
It is easy to verify that the space $B_{X}^{s,q}(\mathbb{R}^{n})$ is Banach spaces and that $\mathcal{S}(\mathbb{R}^{n})\hookrightarrow B_{X}^{s,q}(\mathbb{R}^{n})\hookrightarrow \mathcal{S}^\prime(\mathbb{R}^{n})$.
We denote by $\tilde{B}_{X}^{s,q}(\mathbb{R}^{n})$ the closure of $\mathcal{S}(\mathbb{R}^n)$ in $B_{X}^{s,q}(\mathbb{R}^{n})$.
\end{definition}
\begin{notation}
In the case $X=L^p(\mathbb{R}^n)$ with $1\leq p\leq\infty$, we use the classical notation $B_{p}^{s,q}(\mathbb{R}^{n})$ instead of $B_{L^p(\mathbb{R}^n)}^{s,q}(\mathbb{R}^{n})$.
\end{notation}
An immediate consequence of the definition of the Besov spaces $ B_{X}^{s,q}$ is the following useful result.
\begin{lemma}\label{interp}
Let $X$ be an admissible space.
\begin{enumerate}
    \item If $s_1<s<s_2$ are three real numbers, then
    $$ \left\Vert f\right\Vert _{B_{X}^{s,1}}\lesssim \big(\left\Vert f\right\Vert_{B_{X}^{s_1,\infty}}\big)^{\frac{s_2-s}{s_2-s_1}}\big(\left\Vert f\right\Vert_{B_{X}^{s_2,\infty}}\big)^{\frac{s-s_1}{s_2-s_1}}, $$
$\forall f\in B_{X}^{s_1,\infty}\cap B_{X}^{s_2,\infty}$.
\item If $s_1<0<s_2$ are two real numbers, then
    $$ \left\Vert f\right\Vert _X\lesssim \big(\left\Vert f\right\Vert_{B_{X}^{s_1,\infty}}\big)^{\frac{s_2}{s_2-s_1}}\big(\left\Vert f\right\Vert_{B_{X}^{s_2,\infty}}\big)^{\frac{-s_1}{s_2-s_1}}, $$
$\forall f\in B_{X}^{s_1,\infty}\cap B_{X}^{s_2,\infty}$.
\end{enumerate}
\end{lemma}
\begin{proof}
The second assertion follows immediately from the first one and the continuous embedding
$B_X^{0,1}\hookrightarrow X$ It therefore remains to prove the first assertion. Let $f\in B_{X}^{s_1,\infty}\cap B_{X}^{s_2,\infty}$ and assume that $f\neq0$. Let $N$ be an integer to be chosen later.
\begin{eqnarray*}
     \left\Vert f\right\Vert _{B_{X}^{s,1}}&=&\sum_{j\leq N}2^{sj}\left\Vert \Delta_jf\right\Vert_X+\sum_{j> N}2^{sj}\left\Vert \Delta_jf\right\Vert_X\\
    &\leq&\sum_{j\leq N}2^{(s-s_1)j} \left\Vert f\right\Vert_{B_{X}^{s_1,\infty}}+\sum_{j> N}2^{-(s_2-s)j}\left\Vert f\right\Vert_{B_{X}^{s_2,\infty}}\\
    &\simeq& 2^{(s-s_1)N}\left\Vert f\right\Vert_{B_{X}^{s_1,\infty}}+2^{-(s_2-s)N}\left\Vert f\right\Vert_{B_{X}^{s_2,\infty}}.
\end{eqnarray*}
To conclude, it suffices to choose $N$ such that $2^N\simeq\big(\frac{\left\Vert f\right\Vert_{B_{X}^{s_2,\infty}}}{\left\Vert f\right\Vert_{B_{X}^{s_1,\infty}}}\big)^{\frac{1}{s_2-s_1}}$.
\end{proof}
We now introduce a modified version of Bony's para product decomposition that is better suited to our purposes.
\begin{definition}
For every $f,g\in S^{\prime }(\mathbb{R}^{n}),$ we set
$$\pi^a(f,g)=\sum_{j=3}^\infty S_{j-2}f\Delta_j g,$$
$$\pi^b(f,g)=\sum_{j=-1}^\infty S_{j+2}f\Delta_j g.$$
The identity
$$fg=\pi^a(f,g)+\pi^b(g,f)$$
which holds at least formally, is called the modified Bony para-product decomposition of $fg$, and $\pi^a$ and $\pi^b$ are called the modified Bony para-product operators.
\end{definition}
\begin{remark}\label{spectrum}
 For any $f,g\in \mathcal{S}^{\prime }(\mathbb{R}^{n}),$ and any integer $j$, the spectrum of $S_{j-2}f\Delta_j g$ is contained in the annulus $2^j \mathcal{C}(0,\frac{3}{4},\frac{9}{4})$ centered at the origin with inner radius $2^j \frac{3}{4}$ and out radius $2^j \frac{9}{4}$,  and the spectrum of $S_{j+2}f\Delta_j g$ contained in the ball $2^j B(0,6)$.
\end{remark}
The following lemma collects the main continuity properties of the modified Bony para product operators that will be often used throughout the paper.
\begin{lemma}\label{para}
Let $X$ be an admissible space. Then
\begin{enumerate}
\item For any $s>0,$ the bilinear operator $\pi ^{b}$ is continuous from $%
L^{\infty }(\mathbb{R}^{n})\times B_{X}^{s,\infty }(\mathbb{R}^{n})$
into $B_{X}^{s,\infty }(\mathbb{R}^{n}).$

\item For any real number $s,$ the bilinear operator $\pi ^{a}$ is continuous
from $L^{\infty }(\mathbb{R}^{n})\times B_{X}^{s,\infty }(\mathbb{R}^{n})
$ into $B_{X}^{s,\infty }(\mathbb{R}^{n}).$

\item For any real number $s_{1}<0$ and $s_{2}$ such that $%
s_{1}+s_{2}>0$, the bilinear operator $\pi ^{b}$ is continuous from $%
B_{\infty }^{s_{1},\infty }(\mathbb{R}^{n})\times B_{X}^{s_{2},\infty }(%
\mathbb{R}^{n})$ into $B_{X}^{s_{1}+s_{2},\infty }(\mathbb{R}^{n}).$

\item For any negative real number $s_{1}$ and any real number $s_{2}$, the
bilinear operator $\pi ^{a}$ is continuous from $B_{\infty }^{s_{1},\infty }(%
\mathbb{R}^{n})\times B_{X}^{s_{2},\infty }(\mathbb{R}^{n})$ into $%
B_{X}^{s_{1}+s_{2},\infty }(\mathbb{R}^{n}).$
\end{enumerate}
\end{lemma}
The proof of this lemma follows the same lines as that of \cite[Theorem 4.1]{Lem2}. It relies essentially on Remark \ref{spectrum} together with \cite[Lemma 4.2]{Lem2}.
\par A direct consequence of the two first assertions of the previous lemma is the following important result.
\begin{corollary}\label{banach}
If  $X$ is an admissible space and $s>0$, then
\[
\left\Vert fg\right\Vert_{B_{X}^{s,\infty }}\lesssim \big(\left\Vert f\right\Vert_\infty \left\Vert g\right\Vert_{B_{X}^{s,\infty }}+\left\Vert f\right\Vert_\infty \left\Vert g\right\Vert_{B_{X}^{s,\infty }}\big),
\]
for every $f,g\in L^\infty(\mathbb{R}^n)\cap B_{X}^{s,\infty }(\mathbb{R}^n)$
\end{corollary}
The next lemma summarizes several classical smoothing properties of the heat semigroup $e^{t\Delta}$ in Besov spaces.
\begin{lemma}\label{reg}
Let $X$ be an admissible space and $T\in(0,\infty)$
\begin{enumerate}
\item Let $s>0$ be a real number. Then
\[
\sup_{0<t\leq T}t^{\frac{s}{2}}\left\Vert e^{t\Delta }f\right\Vert
_{X}\lesssim \left\Vert f\right\Vert _{B_{X}^{-s,\infty }},~\forall f\in
B_{X}^{-s,\infty }.
\]%
Moreover, $\lim_{t\rightarrow 0^{+}}t^{\frac{s}{2}}\left\Vert e^{t\Delta
}f\right\Vert _{X}=0$ if $f\in \tilde{B}_{X}^{-s,\infty }.$

\item Let $s_{2}\geq s_{1}$ be two real numbers. Then
\[
\sup_{0<t\leq T}t^{\frac{s_{2}-s_{1}}{2}}\left\Vert e^{t\Delta }f\right\Vert
_{B_{X}^{s_{2},\infty }}\lesssim \left\Vert f\right\Vert
_{B_{X}^{s_{1},\infty }},~\forall f\in B_{X}^{s_{1},\infty }.
\]

\item Let $s\in \mathbb{R}$. Then
\[
\left\Vert \int_{a}^{t}e^{(t-\tau )\Delta }f(\tau )d\tau \right\Vert
_{B_{X}^{s+2,\infty }}\lesssim \sup_{a\leq \tau \leq t}\left\Vert f(\tau
)\right\Vert _{B_{X}^{s,\infty }},
\]%
\[
\left\Vert \int_{a}^{t}e^{(t-\tau )\Delta }f(\tau )d\tau \right\Vert
_{B_{X}^{s+1,\infty }}\lesssim \sqrt{t}\sup_{a\leq \tau \leq t}\left\Vert
f(\tau )\right\Vert _{B_{X}^{s,\infty }},
\]%
for any $0\leq a\leq t\leq T,$ and $f\in L^{\infty }([a,t],B_{X}^{s,\infty
}).$
\end{enumerate}
\end{lemma}
An other extreme important result about the regularizing effect of the heat kernel $e^{t\Delta}$ is the following proposition (see \cite[Theorem 7.3]{Lem2}).
\begin{proposition} [Maximal $L^p(L^q)$ regularity for the heat kernel] \label{max}
Let $T>0 $ and $1<p,q<\infty$. Then
\[
\left\Vert \Delta \int_{0}^{t}e^{(t-\tau )\Delta }f(\tau )d\tau \right\Vert_{L^p([0,T],L^q(\mathbb{R}^n))}\lesssim\left\Vert f \right\Vert_{L^p([0,T],L^q(\mathbb{R}^n))}
\]
for every $f\in L^p([0,T],L^q(\mathbb{R}^n))$.
\end{proposition}
\par\noindent The following result describes the action of homogeneous Fourier multipliers on Besov spaces.
\begin{lemma}\label{Fou}
Let $X$ be an admissible space, $s$ a real number and $1\leq q\leq\infty$. If $\Theta:\mathbb{R}^n_\ast\to\mathbb{R}$ is a $C^\infty$  homogenous function of degree $\alpha>0$ (i.e., $\Theta(\lambda\xi)=\lambda^\alpha\Theta(\xi), \forall \lambda>0, \xi\neq 0$) then the Fourier multiplier $\Theta(D)$ is continuous from $B^{s,q}_X$ to $B^{s-\alpha,q}_X$. In particular, the operator $\mathbb{P}\nabla$ is continuous from $B^{s,q}_X$ to $B^{s-1,q}_X$.
\end{lemma}
We now introduce Lorentz spaces and recall some of their properties that will be needed in the sequel. For a comprehensive treatment and the proofs of the cited results we refer to the two books \cite{BL} and \cite{Lem2}.
\begin{definition}[Lorentz Spaces]
Let $1\leq q\leq \infty $ and $1<p<\infty .$ The Lorentz space $L^{p,q}(%
\mathbb{R}^{n})$ is the interpolation space $[L^{1}\mathbb{(}\mathbb{R}%
^{n}),L^{\infty }\mathbb{(}\mathbb{R}^{n})]_{\theta ,q}$ where $\theta =1-%
\frac{1}{p}.$
\end{definition}
The next proposition collects some useful properties of Lorentz spaces that will be used in sequel of this paper.
\begin{proposition}\label{lorentz} [Some properties of Lorentz spaces]
Let $1<p<\infty $ and $1\leq q\leq \infty .$ The following assertions hold true.
\begin{enumerate}
\item $L^{p,q}(\mathbb{R}^{n})$ is an admissible space.
\item $v\in L^{p,\infty }(\mathbb{R}^{n})$ if and only if $\sup_{\lambda
>0}\lambda ^{p}m(\{x\in \mathbb{R}^{n}:\left\vert v(x)\right\vert \geq
\lambda \})<\infty $ where $m$ is the Lebesgue measure on $\mathbb{R}^{n}$. Moreover
\[
\left\Vert v\right\Vert_{L^{p,\infty }(\mathbb{R}^{n})}\simeq \sup_{\lambda
>0}\lambda [m(\{x\in \mathbb{R}^{n}:\left\vert v(x)\right\vert \geq
\lambda \}]^{1/p}.
\]
\item Let $0<\alpha <n$. n the  function $f(x)=\frac{1}{\left\vert x\right\vert ^{\alpha }%
}$ belongs to the space $L^{\frac{n}{\alpha },\infty }(%
\mathbb{R}^{n}).$

\item $L^{p,1}(\mathbb{R}^{n})\hookrightarrow L^{p,q}(\mathbb{R}%
^{n})\hookrightarrow L^{p,\infty }(\mathbb{R}^{n}).$

\item $L^{p,p}(\mathbb{R}^{n})=L^{p}(\mathbb{R}^{n}).$
\item If $1<p_1<p<p_2<\infty$ then $L^{p_1,\infty}(\mathbb{R}^n)\cap L^{p_2,\infty}(\mathbb{R}^n) \hookrightarrow L^p(\mathbb{R}^{n}).$

\item $\left\Vert f(\lambda .)\right\Vert _{L^{p,q}(\mathbb{R}^{n})}=\lambda
^{-\frac{n}{p}}\left\Vert f\right\Vert _{L^{p,q}(\mathbb{R}^{n})}~\forall
\lambda >0~\forall f\in L^{p,q}(\mathbb{R}^{n}).$

\item $L^{p,\infty }(\mathbb{R}^{n})\hookrightarrow B_{\infty }^{-\frac{n}{p}%
,\infty }(\mathbb{R}^{n}).$

\item If $1<p_{1},p_{2},p_{3}<\infty $ and $1\leq q_{1},q_{2},q_{3}\leq
\infty $ such that $\frac{1}{p_{3}}=\frac{1}{p_{1}}+\frac{1}{p_{2}}$ and $%
\frac{1}{q_{3}}=\frac{1}{q_{1}}+\frac{1}{q_{2}}$ then%
\[
\left\Vert fg\right\Vert _{L^{p_{3},q_{3}}}\leq \left\Vert f\right\Vert
_{L^{p_{1},q_{1}}}\left\Vert g\right\Vert _{L^{p_{2},q_{2}}}\forall f\in
L^{p_{1},q_{1}}(\mathbb{R}^{n}),g\in L^{p_{2},q_{2}}(\mathbb{R}^{n}).
\]

\item If $1<p_{1},p_{2},p_{3}<\infty $ and $1\leq q_{1},q_{2},q_{3}\leq
\infty $ such that $\frac{1}{p_{3}}=\frac{1}{p_{1}}+\frac{1}{p_{2}}-1$ and $%
\frac{1}{q_{3}}=\frac{1}{q_{1}}+\frac{1}{q_{2}}$ then%
\[
\left\Vert f\ast g\right\Vert _{L^{p_{3},q_{3}}}\leq \left\Vert f\right\Vert
_{L^{p_{1},q_{1}}}\left\Vert g\right\Vert _{L^{p_{2},q_{2}}}\forall f\in
L^{p_{1},q_{1}}(\mathbb{R}^{n}),g\in L^{p_{2},q_{2}}(\mathbb{R}^{n}).
\]
\end{enumerate}
\end{proposition}
We now recall a refined version of the  classical Hardy-Littlewood-Sobolev inequality in term of Lorentz spaces.
\begin{lemma}[A refined Hardy-Littlewood Sobolev inequality]\label{hardy}
Let $\alpha \in (0,n)$ and $f\in L^{p_{1},q}(\mathbb{R}^{n})$ with $1<p_{1}<%
\frac{n}{\alpha }$ and $1\leq q\leq \infty .$ Then the function
\[
g(x)=\int_{\mathbb{R}^n}\frac{f(y)}{\left\vert x-y\right\vert ^{n-\alpha }}dy
\]%
belongs to $L^{p_{2},q}(\mathbb{R}^{n})$ and
\[
\left\Vert g\right\Vert _{L^{p_{2},q}}\leq c_{n,\alpha }\left\Vert f\right\Vert
_{L^{p_{1},q}}
\]%
where $c_{n,\alpha }>0$ is an absolute constant that depends only of $n$ and
$\alpha ,$ and $p_2=\frac{n p_1}{n-\alpha p_1}$.
\end{lemma}
\begin{proof}
This is a direct consequence of the convolution estimate  in Lorentz spaces (Proposition \ref{lorentz}) and the fact that the function $h_{\alpha
}(x)=\frac{1}{\left\vert x\right\vert ^{n-\alpha }}$ belongs to the Lorentz
space $L^{\frac{n}{n-\alpha },\infty }(\mathbb{R}^{n})$.
\end{proof}

\begin{corollary}\label{Mon}
Let $f\in L^{n,\infty }(\mathbb{R}^{n})$ and $g\in L^{q}(\mathbb{%
R}^{n})$  with $\frac{n}{n-1}<q<\infty$. Then%
\[
\left\Vert \frac{1}{\sqrt{-\Delta }}(fg)\right\Vert _{q}\lesssim \|f\|_{L^{n,\infty}} \|g\|_{q}.
\]
\end{corollary}

\begin{proof}
First, from the estimate on the point wise product in Lorentz spaces (Proposition \ref{lorentz}) and the fact that $L^q=L^{q,q},$
\[
\|fg\|_{L^{r,q}}\lesssim\|f\|_{L^{n,\infty}}\|g\|_{q},
\]
where $\frac{1}{r}=\frac{1}{q}+\frac{1}{n}.$ The conclusion follows by observing that the Fourier multiplier $\frac{1}{\sqrt{-\Delta}}$ is the convolution operator with the kernel $\frac{c_n}{|x|^{n-1}}$  and applying the preceding lemma.

\end{proof}

\begin{corollary} \label{cor2}
Let $T>0$ and $ 1<\rho _{1}<2.$ Then for every $f\in L^{\rho _{1}}([0,T[;%
\mathbb{R}),$
\[
\left\Vert \int_{0}^{t}\frac{f(s)}{\sqrt{t-s}}ds\right\Vert _{L^{\rho
_{2}}([0,T[,\mathbb{R})}\lesssim \left\Vert f\right\Vert _{L^{\rho _{1}}([0,T[,%
\mathbb{R})}
\]%
where $\rho _{2}=\frac{2\rho _{1}}{2-\rho _{1}}$.
\end{corollary}

\begin{proof}
It suffices to notice that%
\[
\left\Vert \int_{0}^{t}\frac{f(s)}{\sqrt{t-s}}ds\right\Vert _{L^{\rho
_{2}}([0,T[,\mathbb{R})}\leq \left\Vert \int_{\mathbb{R}} \frac{1_{[0,T[}(s)\left\vert f(s)\right\vert}{%
\left\vert t-s\right\vert ^{\frac{1}{2}}}ds\right\Vert _{L^{\rho _{2}}(%
\mathbb{R})}
\]%
and use Lemma \ref{hardy} and the facts $L^{\rho _{2},\rho_1}(%
\mathbb{R})\hookrightarrow L^{\rho _{2}}(%
\mathbb{R})$ and  $L^{\rho _{1},\rho_1}(%
\mathbb{R})=L^{\rho _{1}}(%
\mathbb{R})$
\end{proof}
\par\noindent This is an other simple result about the continuity of the linear integral operator $T$ defined
$$T(f)(t)=\int_{0}^{t}\frac{f(s)}{\sqrt{t-s}}ds.$$
\begin{lemma}\label{con}
Let $T>0$ and $1\leq p\leq\infty$. For every $f\in L^{p}([0,T[;%
\mathbb{R}),$
$$ \left\Vert T(f)\right\Vert_{L^{p}([0,T[,\mathbb{R})}\leq 2\sqrt{T}\left\Vert f\right\Vert_{L^{p}([0,T[,\mathbb{R})}.$$
\end{lemma}
\begin{proof}
It suffices to notice that $T(f)(t)=(\frac{1_{0<s<T}(s)}{\sqrt{s}})\ast(1_{0<s<T}(s)f(s))(t)$ and $\left\Vert \frac{1_{0<s<T}(s)}{\sqrt{s}}\right\Vert_{L^1(\mathbb{R})}=2\sqrt{T}$.
\end{proof}
We prove now a refined Sobolev-type inequality that may be viewed as a non-homogeneous version of the improved Sobolev inequality due to G\'erard, Meyer and Oru \cite{GMO}. For a detailed exposition of several versions of thesis inequality, we refer to \cite[Chapter 2]{BCD}.
\begin{lemma}\label{GMO}
Let $1<p<\infty$. For any $f\in W^{1,p}(\mathbb{R}^n)\cap B^{-1,\infty}_\infty(\mathbb{R}^n)$,
$$\left\Vert f\right\Vert_{2p}\lesssim \left\Vert f\right\Vert_{W^{1,p}}^{\frac{1}{2}}\left\Vert f\right\Vert_{B^{-1,\infty}_\infty }^{\frac{1}{2}} $$
where $\left\Vert f\right\Vert_{W^{1,p}}=\left\Vert f\right\Vert_p +\left\Vert\sqrt{-\Delta} f\right\Vert_p$.
\end{lemma}
\begin{proof}
Let $f\in W^{1,p}(\mathbb{R}^n)\cap B^{-1,\infty}_\infty (\mathbb{R}^n)$ and assume that $f\neq0$. We decompose $f$ into low and high frequencies , $f=f_l+f_h$ where $f_l=S_0f$ and $f_h=\sum_{j=0}^{\infty}\Delta_j f.$ We have
\begin{eqnarray*}
\left\Vert f_l\right\Vert_{2p}&\leq& \left\Vert S_0f\right\Vert_p^{\frac{1}{2}}\left\Vert S_0f\right\Vert_{\infty}^{\frac{1}{2}}\\
&\leq& \left\Vert f\right\Vert_p^{\frac{1}{2}}\left\Vert f\right\Vert_{B^{-1,\infty}_\infty}^{\frac{1}{2}}\\
&\leq& \left\Vert f\right\Vert_{W^{1,p}}^{\frac{1}{2}}\left\Vert f\right\Vert_{B^{-1,\infty}_\infty}^{\frac{1}{2}}.
\end{eqnarray*}%
On the other hand, For any $x\in \mathbb{R}^n$ and any integer $N$,
\begin{eqnarray*}
\vert f_h(x)\vert&\leq& \sum_{j\leq N}\left\Vert\Delta_j f\right\Vert_{\infty}+\sum_{j\geq N+1}2^{-j}\Vert\tilde{\psi}(\frac{D}{2^j})\sqrt{-\Delta}f\Vert\\
&\lesssim&2^N \left\Vert f\right\Vert_{B^{-1,\infty}_\infty}+2^{-N} \mathcal{M}(\sqrt{-\Delta}f)(x),
\end{eqnarray*}%
where $\mathcal{M}(\sqrt{-\Delta}f)$ is the Hardy-Littlewwod maximal function of $\sqrt{-\Delta}f$ \cite{Ste} and $\tilde{\psi}$ is the function defined on $\mathbb{R}^n$ by $\tilde{\psi}(\xi)=\frac{\psi(\xi)}{\vert\xi\vert}$.
\par\noindent Therefore, taking $N$ such that $2^{2N}\simeq\frac{\mathcal{M}(\sqrt{-\Delta}f)(x)}{\left\Vert f\right\Vert_{B^{-1,\infty}_\infty}}$ yields
$$\vert f_h(x)\vert\lesssim (\mathcal{M}(\sqrt{-\Delta}f)(x))^{\frac{1}{2}} \left\Vert f\right\Vert_{B^{-1,\infty}_\infty}^{\frac{1}{2}}.$$
Using now the boundedness of the operator $\mathcal{M}$ on the space $L^p (\mathbb{R}^n)$ \cite{Ste}, we get
$$\left\Vert f_h\right\Vert_{2p}\lesssim \left\Vert \sqrt{-\Delta}f\right\Vert_p^{\frac{1}{2}}\left\Vert f\right\Vert_{B^{-1,\infty}_\infty }^{\frac{1}{2}},$$
which combined with the previous estimate on $\left\Vert f_l\right\Vert_{2p} $, completes the proof.
\end{proof}
We establish now a  simple and useful result that will explain the restriction $p>2$ in the main Theorem \ref{main}.
\begin{lemma}\label{mey}
Let $T>0$ and $2<p<\infty .$ Then the linear operator%
\[
L(v)(t)=\int_{0}^{t}\frac{v(\tau )}{\sqrt{t-\tau }\sqrt{\tau }}d\tau
\]%
is bounded from $L^{p}([0,T],\mathbb{R})$ into itself.
\end{lemma}

\begin{proof}
Let $v\in L^{p}([0,T],\mathbb{R}).$ For any $t\in (0,T],~L(v)(t)=\int_{0}^{1}%
\frac{v(t\tau )}{\sqrt{1-\tau }\sqrt{\tau }}d\tau .$ Then by Minkowski's
inequality, we have
\begin{eqnarray*}
\left\Vert L(v)\right\Vert _{L^{p}([0,T],\mathbb{R})} &\leq &\int_{0}^{1}%
\frac{\left\Vert v(\tau .)\right\Vert _{L^{p}([0,T],\mathbb{R})}}{\sqrt{%
1-\tau }\sqrt{\tau }}d\tau \\
&\leq &\int_{0}^{1}\frac{d\tau }{\sqrt{1-\tau }\tau ^{\frac{1}{2}+\frac{1}{p}%
}}\left\Vert v\right\Vert _{L^{p}([0,T],\mathbb{R})}.
\end{eqnarray*}
\end{proof}

\subsection{Kato  solutions of the Navier-Stokes solutions}
In this subsection, we follow an extended version of Kato's approach to construct regular solutions to the Navier-Stokes equations for initial data belonging to a broad class of functional spaces, which we refer to as Kato spaces. We begin by recalling the definition of a mild solution to the Navier--Stokes equations, originally introduced in \cite{FLT2}.
\begin{definition}
A mild solution to the Navier-Stokes equations (NS) is a function $u\in L^2([0,T],E^2)$ that satisfies the integral equation (\ref{eq2}), where $E^2$ is the space of $v\in L^2_{loc}(\mathbb{R}^3)$ such that $\left\Vert v\right\Vert_{E^2} =\sup_{x_0\in\mathbb{R}^3}\left(\int_{B(x_0,1)}\left\vert v(x)\right\vert^2dx\right)^{\frac{1}{2}}<\infty$, and $ \lim_{\left\vert x_0\right\vert\to\infty}\int_{B(x_0,1)}\left\vert v(x)\right\vert^2dx=0.$
\end{definition}
\begin{remark} The following results were established in \cite{FLT2}; see also \cite{Lem2,May}.
    \begin{enumerate}
        \item In the class $L^2([0,T],E^2)$, the integral equation  (\ref{eq2}) is equivalent to the original Navier-Stokes equations (NS).
        \item If $u\in L^2([0,T],E^2)$ is a mild solution of the Navier-Stokes equations then  $u$ belongs to the space  $ C([0,T],B^{-4,\infty}_\infty(\mathbb{R}^3))$ which justifies the initial value condition $u(0)=u_0$.
        \item The Kernel $\mathcal{K}_t$ (known as the Oseen kernel) of the operator $\mathbb{P}\nabla .e^{t\Delta }$ has the form $\mathcal{K}_t(x)=\frac{1}{t^2}\mathcal{K}(\frac{x}{\sqrt{t}})$ for a smooth function $\mathcal{K}$ such that $(1+\left\vert x\right\vert^4)\mathcal{K}(x)\in L^\infty(\mathbb{R}^3)$. In particular, for all $t>0$
        \begin{equation}\label{kernel}
            \left\Vert \mathcal{K}_t\right\Vert_{ L^1(\mathbb{R}^3)}=\frac{c}{\sqrt{t}},
        \end{equation}
        where $c=\left\Vert \mathcal{K}\right\Vert_{ L^1(\mathbb{R}^3)}>0$.
    \end{enumerate}
\end{remark}
\begin{notation}
In what follows, in order to simplify the notation, we suppress the vector structure of the solutions and the divergence-free condition $\nabla\cdot u=0$. Thus, we regard $u$ as a scalar-valued function and treat the operator $\mathbb{P}\nabla$ as a Fourier multiplier of order one.  We assume moreover  that $B$ is symmetric  and write it in the simplified form
$$
B(u,v)(t)=\int_0^t \mathbb{P}\nabla e^{(t-s)\Delta}(uv)(s) ds.
$$
These simplifications do not affect any of the results or arguments presented in this paper.
\end{notation}
We now introduce the notion of Kato spaces, which extends the class of Lebesgue spaces $L^q(\mathbb{R}^n)$ for $q\geq n$.
\begin{definition}[Kato space]
A Kato  space is an admissible space $X$ (in the sense of Definition \ref{reg}) such
that
\[
\sup_{-1\leq\lambda \leq 1}\lambda \left\Vert f(\lambda .)\right\Vert
_{X}\lesssim \left\Vert f\right\Vert _{X},\forall f\in X.
\]

\end{definition}
\begin{remark}
The Lebesgue spaces $L^p(\mathbb{R}^n)$ with $n\leq p\leq\infty$ and the Lorentz spaces $L^{p,q}(\mathbb{R}^n)$ with $n\leq p<\infty$ and $1\leq q\leq\infty$ are examples of Kato spaces.
\end{remark}
The next result proves that any Kato space is embedded in the limit Besov space $B_{\infty }^{-1,\infty }(\mathbb{R}^{3}).$
\begin{lemma}\label{Linj}
If $X$ is a Kato space then $X\hookrightarrow B_{\infty }^{-1,\infty }(%
\mathbb{R}^{3}).$ More precisely $B^{0,\infty}_X\hookrightarrow B_{\infty }^{-1,\infty }(%
\mathbb{R}^{3}).$
\end{lemma}
\begin{proof}
Let $X$ be a Kato space and $v\in X.$ Let $x\in \mathbb{R}^{n}.$ For any
non negative integer $j,~\Delta _{j}v(x)=\left\langle v(x-2^{-j}.);%
\mathcal{F}^{-1}(\psi )\right\rangle _{S^{\prime }(\mathbb{R}^{n})\times
S(\mathbb{R}^{n})}.$ Hence, there exists a positive constant $C$ that
depends only on $\varphi ,$ such that
\begin{eqnarray*}
\left\vert \Delta _{j}v(x)\right\vert &\leq &C\left\Vert
v(x-2^{-j}.)\right\Vert _{X} \\
& \lesssim& 2^{j}\left\Vert v(x+.)\right\Vert _{X} \\
&=&2^{j}\left\Vert v\right\Vert _{X}.
\end{eqnarray*}%
Similarly, we prove that $\left\vert S_{0}v(x)\right\vert \lesssim
\left\Vert v\right\Vert _{X}.$ This implies
\begin{equation}\label{inj}
\left\Vert
v\right\Vert _{B_{\infty }^{-1,\infty }}\lesssim \left\Vert
v\right\Vert _{X}.
\end{equation}
Now let $w\in B^{0,\infty}_X$. For any integer $j\geq -1$, $2^{-j}\left\Vert\Delta_j w\right\Vert_\infty\simeq \left\Vert\Delta_j w\right\Vert_{B_{\infty }^{-1,\infty }}$. Hence, according to the previous estimate (\ref{inj}), we have
$$2^{-j}\left\Vert\Delta_j w\right\Vert_\infty\lesssim \left\Vert
v_{0}\right\Vert _{X},$$
which implies the required embedding $B^{0,\infty}_X\hookrightarrow B_{\infty }^{-1,\infty }(\mathbb{R}^{n}).$
\end{proof}
The following proposition collects several results concerning the existence, the uniqueness, and the regularity of a class of solutions to the Navier-Stokes equations, as well as a powerful criteria governing the possibility of finite-time blow-up.
\begin{proposition}\label{Ram}
Let $X$ be a Kato space and let $u_0\in\tilde{X}$. there exists a unique maximal solution $u\in C([0,T^\ast),\tilde{X})\cap C_{0,\frac{1}{2}}([0,T^\ast),L^\infty(\mathbb{R}^n))$ to the Navier Stokes equations with initial data $u_0$ where $0<T^\ast\leq \infty$ is the maximal time of existence of the solution $u$. Moreover the following proprieties are satisfied:
\begin{enumerate}
    \item For every $s>0$, $u\in C^\infty((0,T^\ast),B^{s,\infty}_X\cap B^{s,\infty}_\infty).$
    \item It $T^\ast< \infty$ then $u(t)$ does not converge in $B^{-1,\infty}_\infty(\mathbb{R}^n)$ as $t\to T^{\ast -}$.
    \item For every $\alpha>0, u\in C_{0,\frac{\alpha}{2}}([0,T^\ast),B^{\alpha,\infty}_X)$.
\end{enumerate}
Such solution $u$ will be called The $X-$Kato maximal solution to the Navier-Stokes equations associated to the initial data $u_0$.
\end{proposition}
\begin{proof}
   The proofs of the existence and uniqueness statements, together with that of the first assertion, can be found in \cite{May} (see also \cite{Lem2}). The proof of the second assertion follows closely the argument used in the proof of the main theorem in \cite{May2}, where the result is established for $X=L^3(\mathbb{R}^3)$. An English version of that paper can be found in \cite{May3}.  It remains to establish the last assertion, which, to the best of our knowledge, is new and has not previously been established in the literature. In view of the preceding first assertion, we only need to show that for any $\alpha>0$,
    \begin{equation}\label{creg}
        \lim_{t\to 0^+}t^{\frac{\alpha}{2}}\left\Vert u(t)\right\Vert_{B^{\alpha,\infty}_X}=0.
    \end{equation}
Let $\alpha>0$. For every $t\in (0,T^\ast),$
\[
u(t)=e^{t\Delta}u_0 -\int_{0}^{t/2}\mathbb{P}\nabla e^{(t-s)\Delta }(uu)(s)ds-\int_{t/2}^{t}\mathbb{P}\nabla e^{(t-s)\Delta }(uu)(s)ds:=
e^{t\Delta}u_0 -u_l (t)-u_h (t).
\]
Using the second assertion of Lemma \ref{reg}, Lemma \ref{Fou}, and the embedding  $X \hookrightarrow B^{0,\infty}_X$,  we obtain
\begin{eqnarray*}
    \left\Vert u_l(t)\right\Vert_{B^{\alpha,\infty}_X}&\lesssim &\int_0^{t/2} \frac{1}{(t-s)^{\frac{\alpha+1}{2}}} \left\Vert \mathbb{P}\nabla(uu)(s)\right\Vert_{B^{-1,\infty}_X}ds\\
    &\lesssim& \int_0^{t/2} \frac{1}{(t-s)^{\frac{\alpha+1}{2}}} \left\Vert (uu)(s)\right\Vert_Xds\\
    &\lesssim& t^{-\frac{\alpha}{2}} \sup_{0\leq s\leq t}\sqrt{s}\left\Vert u(s)\right\Vert_\infty \sup_{0\leq s\leq t}\left\Vert u(s)\right\Vert_X.
\end{eqnarray*}
On the other hand, using the estimate (\ref{kernel}) and Corollary \ref{creg}, we get
\begin{eqnarray*}
    \left\Vert u_h(t)\right\Vert_{B^{\alpha,\infty}_X}&\lesssim &\int_{t/2}^t \frac{1}{\sqrt{t-s}}\left\Vert (uu)(s)\right\Vert_{B^{\alpha,\infty}_X}ds\\
    &\lesssim& \int_{t/2} ^t \frac{1}{\sqrt{t-s}} \left\Vert u(s)\right\Vert_\infty \left\Vert u(s)\right\Vert_{B^{\alpha,\infty}_X}ds\\
    &\lesssim &t^{-\frac{\alpha}{2}} \sup_{0\leq s\leq t}\sqrt{s}\left\Vert u(s)\right\Vert_\infty \sup_{0\leq s\leq t}s^\frac{\alpha}{2}\left\Vert u(s)\right\Vert_{B^{\alpha,\infty}_X}.
\end{eqnarray*}
Therefore, combining the preceding estimates on $u_l$ and $u_h$, we infer the existence of an absolute constant $C>0$ such that
\begin{eqnarray*}
\sup_{0\leq s\leq t}s^\frac{\alpha}{2}\left\Vert u(s)\right\Vert_{B^{\alpha,\infty}_X}&\leq &\sup_{0\leq s\leq t}s^\frac{\alpha}{2}\left\Vert e^{s\Delta}u_0\right\Vert_{B^{\alpha,\infty}_X}+ C\sup_{0\leq s\leq t}\sqrt{s}\left\Vert u(s)\right\Vert_\infty \sup_{0\leq s\leq t}\left\Vert u(s)\right\Vert_X\\
& & + C \sup_{0\leq s\leq t}\sqrt{s}\left\Vert u(s)\right\Vert_\infty \sup_{0\leq s\leq t}s^\frac{\alpha}{2}\left\Vert u(s)\right\Vert_{B^{\alpha,\infty}_X}.
\end{eqnarray*}
Since $\sqrt{t}\left\Vert u(t)\right\Vert_\infty\to0$ as $t\to 0^+$, $ C \sup_{0\leq s\leq t}\sqrt{s}\left\Vert u(s)\right\Vert_\infty \leq \frac{1}{2}$ for $t$ small enough. Therefore,
\[\sup_{0\leq s\leq t}s^\frac{\alpha}{2}\left\Vert u(s)\right\Vert_{B^{\alpha,\infty}_X}\leq 2 \sup_{0\leq s\leq t}s^\frac{\alpha}{2}\left\Vert e^{s\Delta}u_0\right\Vert_{B^{\alpha,\infty}_X}+ 2C \sup_{0\leq s\leq t}\sqrt{s}\left\Vert u(s)\right\Vert_\infty \sup_{0\leq s\leq t}\left\Vert u(s)\right\Vert_X.
\]
Hence, using the first assertion of Lemma \ref{reg}, we obtain the claimed result (\ref{creg}).
\end{proof}
\section{Proof of the Main Theorem}
The proof of Theorem \ref{main} is greatly inspired by the original paper \cite{Lem}. It is based on some intermediate results that are of independent interest. We begin with a weak-strong uniqueness result.
\begin{theorem}\label{WS}
Let $v_0\in\tilde{L}^{3,\infty}(\mathbb{R}^3)$, $v\in C([0,T^\ast),\tilde{L}^{3,\infty}(\mathbb{R}^3))\cap C_{0,\frac{1}{2}}([0,T^\ast),L^\infty(\mathbb{R}^3))$ be the  $ {L}^{3,\infty}(\mathbb{R}^3)$-Kato maximal solution to the Navier Stokes equations with initial data $v_0$, and $u\in L^p([0,T_1],\tilde{L}^{3,\infty}(\mathbb{R}^3))\cap C([0,T_1],B^{-1,\infty}_\infty(\mathbb{R}^3))$, with $p>2$ and $0<T_1<\infty$, a solution to the Navier Stokes equations associated to the same initial data $v_0$. Then $T^\ast>T_1$ and $u(t)=v(t)$ on $[0,T_1]$.
\end{theorem}
\begin{proof}
Let $0<T<\min{\{T_1,T^\ast\}}.$ First, we may assume, without loss of generality, that
$p\in(2,4)$. Set $w_{1}=B(u,u),w_{2}=B(v,v)$, and $w=w_{1}-w_{2}.$ We shall
prove that $w=0$ on $[0,T]$ which implies $u=v$ on $[0.T]$ since $w$ is also equal to   $u-v$. The proof is divided  into several steps:
\par\noindent\textbf{Step 1:} We shall prove that there exists $q>2$
such that $w_{1},w_{2}\in L^{p}([0,T],L^{q}(\mathbb {R}^{3})).$
\par\noindent Using the estimation (\ref{kernel}) on the Oseen's kernel, the product estimate in Lorentz spaces (Proposition \ref{lorentz}), and the fact that Lorentz spaces are admissible spaces,  we obtain
\[
\left\Vert w_{1}(t)\right\Vert _{L^{\frac{3}{2},\infty }(\mathbb{R}%
^{3})}\lesssim \int_{0}^{t}\frac{1}{\sqrt{t-s}}\left\Vert u(s)\right\Vert
_{L^{3,\infty }(\mathbb{R}^{3})}^{2}ds.
\]%
Therefore, in view of  Corollary \ref{cor2}, we have
\[
\left\Vert w_{1}\right\Vert _{L^{\rho }([0,T],L^{\frac{3}{2},\infty }(%
\mathbb{R}^{3}))}\lesssim \left\Vert u\right\Vert _{L^{p}([0,T],L^{3,\infty
}(\mathbb{R}^{3}))}^{2}
\]%
where $\rho =\frac{2p}{4-p}.$ This implies that $w_1\in L^{p }([0,T],L^{\frac{3}{2},\infty }( \mathbb{R}^{3}))$ since $p\leq \rho$. On the other hand, since $%
e^{t\Delta }v_{0}\in C([0,T],\tilde{L}^{3,\infty }(\mathbb{R}^{3})),$ $%
w_{1}=u-e^{t\Delta }v_{0}\in L^{p}([0,T],\tilde{L}^{3,\infty }(\mathbb{R}%
^{3}))$, then using the continuous embedding of the space  $L^{3,\infty }(\mathbb{R}^{3})\cap L^{%
\frac{3}{2},\infty }(\mathbb{R}^{3})$ into $
L^{q}(\mathbb{R}^{3})$ for any $q\in (\frac{3}{2},3),$  we conclude that $%
w_{1}\in L^{p}([0,T],L^{q}(\mathbb{R}^{3}))$ for any real number $q\in (\frac{3}{2},3).$ The same result
holds true for $w_{2}$, since $v\in C([0,T],\tilde{L}^{3,\infty }(\mathbb{R}%
^{3})\subset L^{p}([0,T],\tilde{L}^{3,\infty }(\mathbb{R}^{3})).$

\par\noindent\textbf{Step 2:} We shall  show that
\begin{equation} \label{claim1}
\lim_{t\to 0^+} B(w,w)(t)=0  \text{ in } B_{\infty }^{-1,\infty }.
\end{equation}
Since
\begin{equation}\label{aicha}
w=u-v\to 0 \text{ in } B_{\infty }^{-1,\infty } \text{ as } t\to 0^+,
\end{equation}
and $w=B(w,w)+2B(v,w)$, it suffices to prove that
\begin{equation}\label{claim}
B(v,w)(t)\to 0  \text{ in } B_{\infty }^{-1,\infty } \text{ as } t\to 0^+   .
\end{equation}
To do this, we use  the Bony-paraproduct operators to split $B(v,w)(t)$ into two parts:
\begin{eqnarray*}
B(v,w)(t) &=&\int_{0}^{t}\mathbb{P}\nabla e^{(t-s)\Delta }\pi
^{a}(v,w)(s)ds+\int_{0}^{t}\mathbb{P}\nabla e^{(t-s)\Delta }\pi ^{b}(w,v)(s)
\\
&=&W_a(t)+W_{b}(t).
\end{eqnarray*}%
Using now the fourth assertion of Lemma \ref{para} and the last assertion of Lemma \ref{reg}, we get
\[
\left\Vert W_a(t)\right\Vert _{B_{\infty }^{-1,\infty }}\lesssim
\sup_{0\leq s\leq t}\left\Vert v(s)\right\Vert _{B_{\infty }^{-1,\infty
}}\sup_{0\leq s\leq t}\left\Vert w(s)\right\Vert _{B_{\infty }^{-1,\infty
}},
\]%
which, combined with (\ref{aicha}),  implies that
\begin{equation}\label{claim2}
\left\Vert W_a(t)\right\Vert _{B_{\infty }^{-1,\infty }}\rightarrow 0\text{
as }t\rightarrow 0.
\end{equation}
Let $1<\alpha <2$ be a fixed real number. Using  Lemma \ref{reg}, Lemma \ref{para}, and Lemma \ref{Fou}, we obtain
\begin{eqnarray*}
\left\Vert W_{b}(t)\right\Vert _{B_{L^{3,\infty }}^{0,\infty }} &\lesssim
&\int_{0}^{t}\frac{1}{(t-s)^{1-\frac{\alpha }{2}}}\left\Vert \mathbb{P}%
\nabla \pi ^{b}(w,v)(s)\right\Vert _{B_{L^{3,\infty }}^{\alpha -2,\infty }}ds
\\
&\lesssim &\int_{0}^{t}\frac{1}{(t-s)^{1-\frac{\alpha }{2}}}\left\Vert \pi ^{b}(w,v)(s)\right\Vert _{B_{L^{3,\infty }}^{\alpha -1,\infty }}ds \\
&\lesssim &\int_{0}^{t}\frac{1}{(t-s)^{1-\frac{\alpha }{2}}}\left\Vert
w(s)\right\Vert _{B_{\infty }^{-1,\infty }}\left\Vert
v(s)\right\Vert _{B_{L^{3,\infty }}^{\alpha ,\infty }}ds \\
&\lesssim &\int_{0}^{t}\frac{1}{(t-s)^{1-\frac{\alpha }{2}}}\frac{1}{s^{%
\frac{\alpha }{2}}}ds\sup_{0\leq s\leq t}\left\Vert
w(s)\right\Vert _{B_{\infty }^{-1,\infty }} \sup_{0\leq s\leq t}s^{\frac{\alpha }{2}}\left\Vert v(s)\right\Vert
_{B_{L^{3,\infty }}^{\alpha ,\infty }}\\
&=&C(\alpha ) \sup_{0\leq s\leq t}\left\Vert w(s)\right\Vert
_{B_{\infty }^{-1,\infty }}\sup_{0\leq s\leq t}s^{\frac{\alpha }{2}}\left\Vert v(s)\right\Vert _{B_{L^{3,\infty
}}^{\alpha ,\infty }}.
\end{eqnarray*}%
Combining the last inequality with the continuous embedding $B_{L^{3,\infty }}^{0,\infty }\hookrightarrow
B_{\infty }^{-1,\infty }$ (see Lemma \ref{Linj}),  we infer that $\left\Vert W_a(t)\right\Vert_{B_{\infty }^{-1,\infty }}\to 0$ as $t\to o^+$ . This combined with (\ref{claim2}) finishes the proof of the required claim (\ref{claim}).
\par\noindent\textbf{Step 3:}
Let $\tau \in (0,T]$, First, by the estimate (\ref{kernel}),
\[
\left\Vert B(w,w)(t)\right\Vert_{
L^{\frac{q}{2}}(\mathbb{R}^{3})}\lesssim \int_{0}^{t}\frac{1}{%
\sqrt{t-s}}\left\Vert w(s)\right\Vert_{L^{q}(\mathbb{R}^{3})}^2ds,
\]%
which implies, thanks to Corollary \ref{cor2},
\[
\left\Vert B(w,w)\right\Vert _{L^\rho([0,\tau ],L^{\frac{q}{2}}(%
\mathbb{R}^{3}))}\lesssim \left\Vert w\right\Vert
_{L^{p}([0,\tau ],L^{q}(\mathbb{R}^{3}))}^{2},
\]%
where $\rho =\frac{2p}{4-p}.$ Since $p\leq\rho$,
the previous estimate gives
\begin{equation}\label{he1}
\left\Vert B(w,w)\right\Vert _{L^p([0,\tau ],L^{\frac{q}{2}}(%
\mathbb{R}^{3}))}\lesssim \left\Vert w\right\Vert
_{L^{p}([0,\tau ],L^{q}(\mathbb{R}^{3}))}^{2}.
\end{equation}
On the other hand,
\[
\sqrt{-\Delta } B(w,w)(t)=\frac{\sqrt{-\Delta }\mathbb{P}\nabla }{\Delta } \Delta \int_{0}^{t} e^{(t-s)\Delta }(w w)(s)ds.
\]
Therefore, From Proposition \ref{max} and the continuity of the operator $%
\frac{\sqrt{-\Delta }\mathbb{P}\nabla }{\Delta }$ on the Lebesgue space $L^{%
\frac{q}{2}}(\mathbb{R}^{3}),$ we deduce that
\begin{equation}\label{he2}
\left\Vert \sqrt{-\Delta }B(w,w)\right\Vert _{L^{\frac{p}{2}}([0,\tau ],L^{%
\frac{q}{2}}(\mathbb{R}^{3}))}\lesssim \left\Vert w\right\Vert
_{L^{p}([0,\tau ],L^{q}(\mathbb{R}^{3}))}^{2}.
\end{equation}
Combining the estimates (\ref{he1}) and (\ref{he2}), we get
\[
\left\Vert B(w,w)\right\Vert _{L^{\frac{p}{2}}([0,\tau ],W^{1,\frac{q}{2}}(%
\mathbb{R}^{3}))}\lesssim \left\Vert w\right\Vert _{L^{p}([0,\tau ],L^{q}(%
\mathbb{R}^{3}))}^{2}.
\]%
Hence, applying Lemma \ref{GMO}, we obtain
\begin{equation}
\left\Vert B(w,w)\right\Vert _{L^{p}([0,\tau ],L^{q}(\mathbb{R}%
^{3}))}\lesssim \left(\sup_{0\leq s\leq \tau }\left\Vert B(w,w)(s)\right\Vert
_{B_{\infty }^{-1,\infty }}\right)^\frac{1}{2}\left\Vert w\right\Vert _{L^{p}([0,\tau ],L^{q}(%
\mathbb{R}^{3}))}.  \label{Est3}
\end{equation}%
Let $\varepsilon >0$ to be fixed later. The density of $L^{\infty }(%
\mathbb{R}^{3})\cap L^{3,\infty }(\mathbb{R}^{3})$ in $\tilde{L}^{3,\infty }(%
\mathbb{R}^{3})$ and the fact that $v\in C([0,T],\tilde{L}^{3,\infty }(%
\mathbb{R}^{3}))$ ensure the existence of $v_{\varepsilon }\in
C([0,T],L^{3,\infty }(\mathbb{R}^{3}))$ and $~v_{\infty }\in
C([0,T],L^{\infty }(\mathbb{R}^{3}))$ such that
\[
v=v_{\varepsilon }+v_{\infty },\sup_{0\leq s\leq T}\left\Vert v_{\varepsilon
}(s)\right\Vert _{L^{3,\infty }}\leq \varepsilon \text{ and }\sup_{0\leq
s\leq T}\left\Vert v_{\varepsilon }(s)\right\Vert _{\infty }=M_{\infty
}<\infty .
\]%
We decompose $B(v,w)$ into two terms as follows
\begin{eqnarray*}
B(v,w)(t)&=&
\int_{0}^{t} \mathbb{P}\nabla e^{(t-s)\Delta }(v_{\varepsilon
}w)(s)ds+\int_{0}^{t}\mathbb{P}\nabla e^{(t-s)\Delta }(v_{\infty
}w)(s)ds\\
&=& \frac{\mathbb{P}\nabla\sqrt{-\Delta }}{\Delta }
\Delta \int_{0}^{t}e^{(t-s)\Delta }\frac{1}{\sqrt{-\Delta }}(v_{\varepsilon
}w)(s)ds+\int_{0}^{t}\mathbb{P}\nabla e^{(t-s)\Delta }(v_{\infty
}w)(s)ds\\
&=& \Sigma_{\varepsilon}(t)+\Sigma_{\infty }(t).
\end{eqnarray*}
By applying, respectively,  the continuity of the operator $\frac{\mathbb{P}\nabla\sqrt{-\Delta }}{\Delta}$ on $ L^{q}(\mathbb{R}^{3})$, Proposition \ref{max} and Corollary \ref{Mon},  we obtain
\begin{eqnarray}
\left\Vert \Sigma_{\varepsilon}\right\Vert _{L^{p}([0,\tau ],L^{q}(\mathbb{R}%
^{3}))} &\lesssim &\sup_{0\leq s\leq \tau }\left\Vert v_{\varepsilon
}(s)\right\Vert _{L^{3,\infty }}\left\Vert w\right\Vert _{L^{p}([0,\tau
],L^{q}(\mathbb{R}^{3}))}  \nonumber \\
&\lesssim &\varepsilon \left\Vert w\right\Vert _{L^{p}([0,\tau ],L^{q}(%
\mathbb{R}^{3}))}.  \label{Est2}
\end{eqnarray}%
On the other hand, using as usual the estimate (\ref{kernel}) on the Ossen kernel, we get
\[\left\Vert \Sigma_{\infty }(t)\right\Vert _{L^{q}(\mathbb{R}^{3})} \lesssim
\int_{0}^{t}\frac{1}{\sqrt{t-s}}M_{\infty }\left\Vert w(s)\right\Vert
_{L^{q}(\mathbb{R}^{3})}ds
\]
Therefore, from Lemma \ref{con}, we have
\begin{equation}\label{Est0}
\left\Vert \Sigma_{\infty }\right\Vert _{L^{p}([0,\tau ],L^{q}(\mathbb{R}^{3}))}\lesssim M_{\infty }\sqrt{\tau }\left\Vert w\right\Vert_{L^{p}([0,\tau ],L^{q}(\mathbb{R}^{3}))}
\end{equation}
Finally, going back to the identity $w=B(w,w)+2B(v,w)$ and gathering the
estimates (\ref{Est3}), (\ref{Est2}), and (\ref{Est0}), we infer that
there exists an absolute constant $C>0$ (independent of $\varepsilon $ and $%
\tau $) such that for any $\tau \in \lbrack 0,T]$%
\[
\left\Vert w\right\Vert _{L^{p}([0,\tau ],L^{q}(\mathbb{R}^{3}))}\leq
C\big((\sup_{0\leq s\leq \tau }\left\Vert B(w,w)(s)\right\Vert
_{B_{\infty }^{-1,\infty }})^\frac{1}{2}+\varepsilon +M_{\infty }\sqrt{\tau }\big) \left\Vert
w\right\Vert _{L^{p}([0,\tau ],L^{q}(\mathbb{R}^{3}))}.
\]%
Hence, using the estimate (\ref{claim1}) and  choosing respectively $\varepsilon $ and $\tau $ small enough, we conclude that there exists $\tau \in (0,T]$ such
that
\[
\left\Vert w\right\Vert _{L^{p}([0,\tau ],L^{q}(\mathbb{R}^{3}))}\leq \frac{1%
}{2}\left\Vert w\right\Vert _{L^{p}([0,\tau ],L^{q}(\mathbb{R}^{3}))}
\]%
which implies that $w(t)=0$ on $[0,\tau ]$ and by consequent $u(t)=v(t)$ on $%
[0,\tau ].$ We shall now apply a well-known iteration argument to infer that  $u(t)=v(t)$ on the entire interval $[0,T].$ Let $t_\ast=\sup\{0\leq t\leq T: u=v  \text{ on } [0,t]\}$. We have $t_\ast\in[\tau,T].$ If $t_\ast<T$, then $u(.+t_\ast)$ and $v(.+t_\ast)$ are two mild solutions on $[0,T-t_\ast]$ to the Navier-Stokes equations associated to the same initial data $u(t_\ast)=v(t_\ast)$. Hence, applying what we have already done for $u$ and $v$ ensures the existence of $\tau_\ast\in(0,T-t_\ast)$ such that $u(.+t_\ast)=v(.+t_\ast)$ on the interval $[0,\tau_\ast ]$. This contradicts the definition of $t_\ast$. We then conclude that $u(t)=v(t)$ on $[0,T]$. Since the argument holds for every  $ T\in (0,\min\{T^\ast,T_1\}),$ we infer  that $u(t)=v(t)$ on the greater interval $[0,\min\{T^\ast,T_1\}).$ Finally, necessary  $T^\ast> T_1$, because if $T^\ast\leq T_1$ then
$$\lim_{t\to T^{\ast -}}v(t)=\lim_{t\to T^{\ast -}}u(t)=u(T^\ast)$$
in $B^{-1,\infty}_\infty(\mathbb{R}^3)$, which is impossible in view of  the last assertion of Proposition \ref{Ram}. This completes  the proof.
\end{proof}
The next theorem provides a control on the growth of $\left\Vert u(t)\right\Vert_\infty$  as $t$ approaches $0$ for regular solutions $u$ to the  Navier-Stokes equations that in addition belong to  the space $C([0,T],\tilde{B}^{-1,\infty}_\infty(\mathbb{R}^3))$.
\begin{theorem}\label{lem}
Let $T>0$ and $\gamma >1$ be two real numbers. If $u\in C([0,T],\tilde{B}%
_{\infty }^{-1,\infty }(\mathbb{R}^{3}))\cap C((0,T),B_{\infty }^{\gamma
,\infty }(\mathbb{R}^{3}))$ is a mild solution of the Navier Stokes
equations, then
\[
u\in  C_{0,\frac{\gamma+1}{2}}([0,T],B^{\gamma,\infty}_\infty(\mathbb{R}^3))\cap C_{0,\frac{1}{2}}([0,T],L^\infty(\mathbb{R}^3)).
\]
\end{theorem}
\begin{proof}
Let $\varepsilon >0$ be a fixed positive real number. Since $u\in C([0,T],%
\tilde{B}_{\infty }^{-1,\infty }(\mathbb{R}^{3})),$ there exist $%
u_{\varepsilon }$ and $u_{\infty }$ in $C([0,T],\tilde{B}_{\infty
}^{-1,\infty }(\mathbb{R}^{3}))$ such that%
\[
u=u_{\varepsilon }+u_{\infty },\sup_{0\leq s\leq T}\left\Vert u_{\varepsilon
}(s)\right\Vert _{B_{\infty }^{-1,\infty }}\leq \varepsilon \text{ and }%
\sup_{0\leq s\leq T}\left\Vert u_{\infty }(s)\right\Vert _{\infty
}=N_{\infty }<\infty .
\]%
Let $(t_{n})\in (0,\frac{T}{2}]$ be a sequence converging to $0.$ We set
$u_{n}(t)=u(t+t_{n})$ for any $n\in \mathbb{N}$ and $t\in \lbrack 0,\frac{T}{%
2}].$ Since $u_{n}$ is a  mild solution on $[0,\frac{T}{2}]$ to the Navier
Stokes equations, then%
\begin{eqnarray}
u_{n}(t) &=&e^{\frac{t}{2}\Delta }u_{n}(\frac{t}{2})-\int_{\frac{t}{2}}^{t}%
\mathbb{P}\nabla e^{(t-s)\Delta }u_{n}(s)u_{n}(s)ds  \nonumber \\
&=&e^{\frac{t}{2}\Delta }u_{n}(\frac{t}{2})-\sum_{\nu \in \{a,c\}}\int_{%
\frac{t}{2}}^{t}\mathbb{P}\nabla e^{(t-s)\Delta }\pi ^{\nu
}(u_{n}(s),u_{n}(s))ds  \nonumber \\
&=&e^{\frac{t}{2}\Delta }u_{\varepsilon n}(\frac{t}{2})+e^{\frac{t}{2}\Delta
}u_{\infty n}(\frac{t}{2})-\sum_{\nu \in \{a,b\}}\int_{\frac{t}{2}%
}^{t}e^{(t-s)\Delta }\mathbb{P}\nabla \pi ^{\nu }(u_{\varepsilon
n}(s),u_{n}(s))ds  \nonumber \\
&&-\sum_{\nu \in \{a,b\}}\int_{\frac{t}{2}}^{t}e^{(t-s)\Delta }\mathbb{P}%
\nabla \pi ^{\nu }(u_{\infty n}(s),u_{n}(s))ds,  \label{IDE}
\end{eqnarray}%
where $u_{\varepsilon n}(t)=u_{\varepsilon }(t+t_{n})$ and $u_{\infty
n}(t)=u_{\infty }(t+t_{n}).$ We now introduce  the sequence of functions $%
\Theta _{n}:[0,\frac{T}{2}]\rightarrow \mathbb{R}$ defined by%
\[
\Theta _{n}(t)=\sup_{0\leq s\leq t}s^{\frac{\gamma +1}{2}}\left\Vert
u_{n}(s)\right\Vert _{B_{\infty }^{\gamma ,\infty }}.
\]%
Applying the second assertion of Lemma \ref{reg}, we obtain%
\begin{eqnarray}
\left\Vert e^{\frac{t}{2}\Delta }u_{\varepsilon n}(\frac{t}{2})\right\Vert
_{B_{\infty }^{\gamma ,\infty }} &\lesssim &t^{-\frac{\gamma +1}{2}%
}\left\Vert u_{\varepsilon n}(\frac{t}{2})\right\Vert _{B_{\infty
}^{-1,\infty }}  \nonumber \\
&\lesssim &\varepsilon ~t^{-\frac{\gamma +1}{2}},  \label{M1}
\end{eqnarray}%
and
\begin{eqnarray}
\left\Vert e^{\frac{t}{2}\Delta }u_{\infty n}(\frac{t}{2})\right\Vert
_{B_{\infty }^{\gamma ,\infty }} &\lesssim &t^{-\frac{\gamma }{2}}\left\Vert
u_{\varepsilon n}(\frac{t}{2})\right\Vert _{B_{\infty }^{0,\infty }}
\nonumber \\
&\lesssim &t^{-\frac{\gamma }{2}}\left\Vert u_{\varepsilon n}(\frac{t}{2}%
)\right\Vert _{\infty }  \nonumber \\
&\lesssim &N_{\infty }\text{ }t^{-\frac{\gamma }{2}}.  \label{M3}
\end{eqnarray}%
On the other hand, using again Lemma \ref{reg} with Lemma \ref{Fou} and Lemma \ref{para}, we get
\begin{eqnarray}
\left\Vert \int_{\frac{t}{2}}^{t}e^{(t-s)\Delta }\mathbb{P}\nabla \pi ^{\nu
}(u_{\varepsilon n}(s),u_{n}(s))ds\right\Vert _{B_{\infty }^{\gamma ,\infty
}} &\lesssim &\sup_{\frac{t}{2}\leq s\leq t}\left\Vert \mathbb{P}\nabla \pi
^{\nu }(u_{\varepsilon n}(s),u_{n}(s))\right\Vert _{B_{\infty }^{\gamma
-2,\infty }}  \nonumber \\
&\lesssim &\sup_{\frac{t}{2}\leq s\leq t}\left\Vert \pi ^{\nu
}(u_{\varepsilon n}(s),u_{n}(s))\right\Vert _{B_{\infty }^{\gamma -1,\infty
}}  \nonumber \\
&\lesssim &\sup_{\frac{t}{2}\leq s\leq t}\left\Vert u_{\varepsilon
n}(s)\right\Vert _{B_{\infty }^{-1,\infty }}\left\Vert u_{n}(s)\right\Vert
_{B_{\infty }^{\gamma ,\infty }}  \nonumber \\
&\lesssim &\varepsilon ~t^{-\frac{\gamma +1}{2}}~\Theta _{n}(t).  \label{M4}
\end{eqnarray}%
Moreover, invoking estimate (\ref{kernel}) and Lemma \ref{para}, yields
\begin{eqnarray}
\left\Vert \int_{\frac{t}{2}}^{t}e^{(t-s)\Delta }\mathbb{P}\nabla \pi ^{\nu
}(u_{\infty n}(s),u_{n}(s))ds\right\Vert _{B_{\infty }^{\gamma ,\infty }}
&\lesssim &\int_{\frac{t}{2}}^{t}\frac{1}{\sqrt{t-s}}\left\Vert \pi ^{\nu
}(u_{\infty n}(s),u_{n}(s))\right\Vert _{B_{\infty }^{\gamma ,\infty }}ds
\nonumber \\
&\lesssim &\int_{\frac{t}{2}}^{t}\frac{1}{\sqrt{t-s}}\left\Vert
u_{\varepsilon n}(s)\right\Vert _{\infty }\left\Vert u_{n}(s)\right\Vert
_{B_{\infty }^{\gamma ,\infty }}ds  \nonumber \\
&\lesssim &N_{\infty }\Theta _{n}(t)\int_{\frac{t}{2}}^{t}\frac{1}{\sqrt{t-s}%
}\frac{1}{s^{\frac{\gamma +1}{2}}}ds  \nonumber \\
&\simeq &t^{-\frac{\gamma }{2}}N_{\infty }\Theta _{n}(t).  \label{M5}
\end{eqnarray}%
Collecting the estimates  (\ref{IDE}), (\ref{M1}), (\ref{M3}), (\ref{M4}), and (\ref%
{M5}),  we deduce that there exists a constant $C>0$ independent of $%
t,\varepsilon ,$ and $n,$ such that%
\begin{equation}\label{bas}
\Theta _{n}(t)\leq C\left( \left[\varepsilon +N_{\infty }\sqrt{t}\right]+\left[
\varepsilon +N_{\infty }\sqrt{t}\right] \Theta _{n}(t)\right) .
\end{equation}
Let  us assume now that $\varepsilon\leq\frac{1}{4C}$ and choose  $\delta >0$ small enough such that
\[
(\varepsilon +N_{\infty }\sqrt{\delta})\leq \frac{1}{2C}.
\]
Therefore, the inequality  (\ref{bas}) implies
\[
\Theta _{n}(\delta )\leq 2C\left( \varepsilon +N_{\infty }\sqrt{\delta}\right) .
\]%
Letting now $n\rightarrow \infty ,$ we conclude that %
\[
\sup_{0\leq t\leq \delta }t^{\frac{1+\gamma }{2}}\left\Vert u(t)\right\Vert
_{B_{\infty }^{\gamma ,\infty }}\leq 2C\left( \varepsilon +N_{\infty }\sqrt{\delta}\right) .
\]%
Since $\varepsilon $ is an arbitrary positive real number, the last inequality implies%
\begin{equation}\label{Aicha}
\lim_{t\rightarrow 0}t^{\frac{1+\gamma }{2}}\left\Vert u(t)\right\Vert
_{B_{\infty }^{\gamma ,\infty }}=0.
\end{equation}
Using now the fact that $ u\in L^\infty([0,T],B_{\infty }^{-1,\infty })$ and  the interpolation inequality (see Lemma \ref{interp})
\[
\left\Vert f\right\Vert _{\infty }\lesssim \left( \left\Vert f\right\Vert
_{B_{\infty }^{-1,\infty }}\right) ^{\frac{\gamma }{\gamma +1}}\left(
\left\Vert f\right\Vert _{B_{\infty }^{\gamma ,\infty }}\right) ^{\frac{1}{%
\gamma +1}},
\]%
we deduce the estimate
\[
\sqrt{t}\left\Vert u(t)\right\Vert _{\infty }\lesssim \left( \left\Vert u(t)\right\Vert
_{B_{\infty }^{-1,\infty }}\right) ^{\frac{\gamma }{\gamma +1}}\left(t^{\frac{1+\gamma}{2}}
\left\Vert u(t)\right\Vert _{B_{\infty }^{\gamma ,\infty }}\right) ^{\frac{1}{%
\gamma +1}}
\]
that, thanks to (\ref{Aicha}), yields  the desired conclusion
\[
\sqrt{t}\left\Vert u(t)\right\Vert _{\infty }\rightarrow 0\text{ as }%
t\rightarrow 0.
\]
The proof is then achieved.
\end{proof}
We are now in a position to prove the main result of this paper, namely Theorem \ref{main}.
\begin{proof}
Let $u\in L^p([0,T],\tilde{L}^{3,\infty}(\mathbb{R}^3))\cap C([0,T],B^{-1,\infty}_\infty(\mathbb{R}^3)) $ be a solution to the Navier-Stokes equations. We will first prove the regularity properties of $u$. For every $\nu \in (0,T)$ there exists $t_{0}\in (0,\nu) $ such
that $v_{0}:=u(t_{0})$ belongs to the Kato space $\tilde{L}^{3,\infty }(\mathbb{R}^{3}).$ Let $v$ be the
maximal $\tilde{L}^{3,\infty }(\mathbb{R}^{3})$- Kato solution of the
Navier-Stokes associated to the initial data $v_{0}$ and let $T^{\ast }\in
(0,\infty ]$ be its maximal time of existence. Since $ u(.+t_0)$ is also a solution of the same Navier-Stokes equations, then according to Theorem \ref{WS}, $ T^\ast>T-t_0,
u(t+t_{0})=v(t)$ on the interval $[0,T-t_{0}].$ Hence, by
using the regularity properties of Kato solutions (see Proposition \ref{Ram}) and the fact that $t_0$ is arbitrary small, we infer that
\[u\in C^\infty((0,T^\ast),B^{s,\infty}_X\cap B^{s,\infty}_\infty), ~\forall s>0;\]
which implies in particular that $u\in C^\infty((0,T]\times\mathbb{R}^3).$
On the other hand since $\tilde{L}^{3,\infty}(\mathbb{R}^3)\hookrightarrow \tilde{B}_{\infty }^{-1,\infty }$, $u(t)\in \tilde{B}_{\infty }^{-1,\infty }$ for almost every $t\in[0,T]$. By consequent,  $u\in C([0,T],\tilde{B}_{\infty }^{-1,\infty })$. Therefore,  in view
of Theorem \ref{lem},
\[
u\in  C_{0,\frac{s+1}{2}}([0,T],B^{s,\infty}_\infty(\mathbb{R}^3))\cap C_{0,\frac{1}{2}}([0,T],L^\infty(\mathbb{R}^3)),~\forall s>0.
\]
\par\noindent Proceeding exactly as in the proof of Theorem \ref{lem}, we obtain  $u\in  C_{0,\frac{s}{2}}([0,T];B^{s,\infty}_{L^{3,\infty}(\mathbb{R}^3)})$ for any $s>0.$ Hence, Using the first interpolation inequality in Lemma \ref{interp}, we infer that for any positive real number $\alpha$, $u\in  C_{0,\frac{\alpha}{2}}([0,T],B^{\alpha,1}_{L^{3,\infty}(\mathbb{R}^3)})$. Thanks to boundedness of the operator $\sqrt{-\Delta}^\alpha$ from $B ^{\alpha,1}_{L^{3,\infty}(\mathbb{R}^3)}$ in to  $ B ^{0,1}_{L^{3,\infty}(\mathbb{R}^3)} $ (Lemma \ref{Fou}) and the continuous embedding $ B ^{0,1}_{L^{3,\infty}(\mathbb{R}^3)} \hookrightarrow L^{3,\infty}(\mathbb{R}^3)$ , we deduce that  $ \sqrt{-\Delta}^\alpha u\in C_{0,\frac{\alpha}{2}}([0,T], L^{3,\infty}(\mathbb{R}^3)).$
\par\noindent It remains to prove the uniqueness assertion in Theorem \ref{main}. Let $u_1$ and $u_2$ be two solutions of the Navier--Stokes equations with the same initial data $u_0$, both belonging to the space $%
L^{p}([0,T],\tilde{L}^{3,\infty }(\mathbb{R}^{3}))\cap C([0,T],B_{\infty
}^{-1,\infty }(\mathbb{R}^{3}))$. The function $w:=u_{2}-u_{1}$
satisfies
\begin{eqnarray*}
w &=&B(u_{2},u_{2})-B(u_{1},u_{1}) \\
&=&B(u_{2},w)+B(w,u_{1}).
\end{eqnarray*}%
Hence, using the estimate (\ref{kernel}) and the fact that Lorentz spaces are admissible spaces,  we obtain
\begin{eqnarray*}
\left\Vert w(t)\right\Vert _{L^{3,\infty }} &\lesssim &\int_{0}^{t}\frac{1}{%
\sqrt{t-s}\sqrt{s}}\left( \sqrt{s}\left\Vert u_{2}(s)\right\Vert _{\infty }+%
\sqrt{s}\left\Vert u_{1}(s)\right\Vert _{\infty }\right) \left\Vert
w(s)\right\Vert _{L^{3,\infty }}ds \\
&\lesssim &\sigma (t)\int_{0}^{t}\frac{1}{\sqrt{t-s}\sqrt{s}}\left\Vert
w(s)\right\Vert _{L^{3,\infty }}ds,
\end{eqnarray*}%
where $\sigma(t)=\sup_{0\leq s\leq t}\left( \sqrt{s}\left\Vert
u_{2}(s)\right\Vert _{\infty }+\sqrt{s}\left\Vert u_{1}(s)\right\Vert
_{\infty }\right) .$ Invoking Lemma \ref{mey}, we infer
that for that for any $\delta \in (0,T],$%
\[
\left\Vert w\right\Vert _{L^{p}([0,\delta ],L^{3,\infty })}\lesssim \sigma
(\delta )\left\Vert w\right\Vert _{L^{p}([0,\delta ],L^{3,\infty })}.
\]%
This, combined with the fact that $\lim_{t\rightarrow 0}\sigma (t)=0$, ensures
the existence of $\delta \in (0,T]$ such that $\left\Vert w\right\Vert
_{L^{p}([0,\delta ],L^{3,\infty })}\lesssim \frac{1}{2}\left\Vert
w\right\Vert _{L^{p}([0,\delta ],L^{3,\infty })}.$ Therefore $w=0$ on $%
[0,\delta ]$ and by consequent $u_{1}=u_{2}$ on $[0,\delta ].$ By the standard  iteration process,
we conclude that $u_{1}=u_{2}$ on the entire interval $[0,T].$ This completes the proof.

\end{proof}

\end{document}